\documentclass[11pt]{article}
\usepackage{amsfonts}
\usepackage{amsmath, amsfonts, amsthm, amssymb, amscd, mathrsfs}
\usepackage[mathcal]{euscript}

\usepackage{color}
\usepackage {a4wide}
\theoremstyle{plain}
\newtheorem{theorem}{Theorem}[section]
\newtheorem{proposition}[theorem]{Proposition}

\newtheorem{lemma}[theorem]{Lemma}
\newtheorem{corollary}[theorem]{Corollary}
\newtheorem{definition}[theorem]{Definition}
\newtheorem{remark}[theorem]{Remark}
\newtheorem{example}{Example}

\newtheorem{claim}{Claim}
\newtheorem{Theorem}{Theorem}
\numberwithin{equation}{section}
\let\oldmarginpar\marginpar
\renewcommand\marginpar[1]{\-\oldmarginpar[\raggedleft\footnotesize #1]%
{\raggedright\footnotesize #1}}

\usepackage{a4wide}

\newcommand \bei {\begin{itemize}}
\newcommand \eei {\end{itemize}}

\newcommand \be         {\begin{equation}}
\newcommand \bel {\be\label}

\newcommand \eps \epsilon

\newcommand \coeff \kappa

\newcommand \del \partial

\newcommand \ee         {\end{equation}}
\newcommand \la \langle
\newcommand \ra \rangle

\definecolor{myred}{rgb}{0.858, 0.0, 0.0}

\begin{document}

\title{Progress on nonuniqueness of solutions to the vacuum Einstein conformal constraint equations with positive Yamabe invariant
}

\author{The Cang Nguyen\footnote{
Department of Mathematics, University of Quy Nhon, Vietnam.
E-mail: {\sl alpthecang@gmail.com.}
}}
\date{\today}
\maketitle
\begin{abstract}
	In this article, we make a generalization of classical fixed point theorems by using the concept of half-continuity and then apply it to improve the nonuniqueness result for solutions to the vacuum Einstein conformal equations shown in \cite{NguyenNonexistence}.
\end{abstract} 
\section{Introduction}
On a given smooth, compact $n-$manifold $M$ with $n\ge 3$, the \textit{vacuum Einstein constraint equations} for a metric $\widetilde{g}$ and a symmetric $(0,2)-$tensor $\widetilde{K}$ are
\begin{equation}\label{contraint equations}
\aligned
R_{\widetilde{g}}-|\widetilde{K}|_{\widetilde{g}}^2+(\text{Tr}_{\widetilde{g}}\widetilde{K})^2&=0
\\
\text{div}_{\widetilde{g}}\widetilde{K}-d(\text{Tr}_{\widetilde{g}}\widetilde{K})&=0,
\endaligned
\end{equation}
where $R_{\widetilde{g}}$ is the scalar curvature of $\widetilde{g}$. The study of solutions to \eqref{contraint equations} is a topical issue because they can
be used to produce solutions of the Einstein equations on a Lorentzian $(n + 1)-$dimensional manifold, as guaranteed by a well known result of Choquet-Bruhat \cite{CB1}.     

One of most efficient approaches to solving \eqref{contraint equations} is the conformal method introduced by Lichnerowicz \cite{Li44} and later Choquet-Bruhat and York \cite{CBY80}. The idea of this method is to divide a solution $(\widetilde{g},\,\widetilde{K})$ into some reasonable parts, and then solve for the rest of the data. More precisely, given seed data:
\begin{itemize}
	\item $g$ - a Riemannian metric on $M$,
	\item $\sigma$ - a divergence-free ($\nabla^i\sigma_{ij}=0$), trace-free ($g^{ij}\sigma_{ij}=0$) symmetric tensor,
	\item $\tau$ - a scalar field,
\end{itemize} 
we seek a solution $(\widetilde{g},\,\widetilde{K})$ of the form
$$
\aligned
\widetilde{g} =&\varphi^{N-2}g
\\
\widetilde{K}=&\frac{\tau}{n}\varphi^{N-2}g+\varphi^{-2}(\sigma+LW).
\endaligned
$$
Here a positive function $\varphi$ and a $1-$form $W$ are unknowns, $N=\frac{2n}{n-2}$ and $L$ is the conformal Killing operator defined by
$$
(LW)_{ij}=\nabla_{i} W_{j}+\nabla_{j}W_{i}-\frac{2}{n}(\text{div}W)g_{ij}.
$$
Equations \eqref{contraint equations} then become a nonlinear elliptic system for $\varphi$ and $W$:
 \begin{subequations}\label{CE}
 \begin{eqnarray}
 \label{eqLichnerowicz}
 \frac{4(n-1)}{n-2}\Delta_g \varphi+R_g\varphi&=&-\frac{n-1}{n}\tau^2\varphi^{N-1}+|\sigma+LW|^{2}\varphi^{-N-1}{\footnotesize [\text{Lichnerowicz equation}]}
 \\
 \label{eqVector}
 -\frac{1}{2}L^{*}LW&=&\frac{n-1}{n}\varphi^{N}d\tau, \qquad\qquad\qquad\qquad\quad\qquad{\footnotesize [\text{vector equation}]},
 \end{eqnarray}
 \end{subequations}
 where $\Delta_g$ is the negative Laplace operator and $L^{*}$ is the formal $L^{2}-$adjoint of $L$. These coupled equations are called the \textit{vacuum Einstein conformal constraint equations}, or simply \textit{the conformal equations}.

When the mean curvature $\tau$ is constant or almost constant, a completed description of this system is given by many authors. A solution $(\varphi,\,W)$ to \eqref{CE}, if it exists, is unique in this situation. The interested reader is referred to \cite{ACI08,BartnikIsenberg,Isenberg,MaxwellRoughCompact} for further information. When $\tau$ is freely specified, the situation appears much harder and only two methods exist in \cite{HNT2, MaxwellNonCMC} and \cite{DahlGicquaudHumbert} to tackle this case. However neither of them tells us whether solutions are unique or not.
 
Using the technique in \cite{DahlGicquaudHumbert}, recent work of the author in \cite{NguyenNonexistence} shows that if the Yamabe invariant is positive and if
\begin{equation}\label{key condition 1}
\bigg|L\bigg(\frac{d\tau}{\tau}\bigg)\bigg|\le c\bigg|\frac{d\tau}{\tau}\bigg|^2
\end{equation}
for some constant $c>0$,  the system \eqref{CE} associated with data $(g,t_i\tau^a,k\sigma)$ has at least two solutions provided that
\begin{itemize} 
	\item[(i)] $\sigma\ne 0$ and $\text{supp}\{\sigma\}\subsetneq M\setminus U$, for some neighborhood $U$ of the critical set of $\tau$,
	\item[(ii)] $a,\,k$ are sufficiently large constants only depending on $(g,\tau,\sigma,c)$,
	\item[(iii)] $\{t_i\}$ is a certain real sequence converging to $0$. 
\end{itemize}

 In this article, we are interested in this nonuniqueness result. The question we shall be concerned with is whether we can optimize the assumptions (i--iii), for instance with an arbitrary $\sigma\ne 0$ and for all sufficiently small $t$. Note that the Schauder fixed point theorem used in most of the relevant articles seems not to be helpful in this situation, so we will employ here the so-called half-continuity method previously introduced in \cite{Nguyen} to address the question. Explicitly, the main result we would like to obtain is as follows.
 \begin{Theorem}\label{theorem.main}
 Let $g\in W^{2,p}$ with $p > n$ be a Yamabe-positive metric on a smooth compact $n-$manifold. Assume that $g$ has no conformal Killing vector field, $\sigma \in W^{1,p}\setminus \{0\}$ and $\tau\in W^{1,p}$ does not change sign. Assume furthermore that $\tau$ satisfies 
 $$
 \bigg|L\bigg(\frac{d\tau}{\tau}\bigg)\bigg|\le c\bigg|\frac{d\tau}{\tau}\bigg|^2
 $$
 for some constant $c>0$.  Then the two following assertions are true:
 \begin{itemize}
 	\item[1.] Given $a>\frac{c}{2} \sqrt{\dfrac{n}{n-1}}$ the system \eqref{CE} associated with $(g,t\tau^a,\sigma)$ has at least two solutions for all $t>0$ small enough only depending on $(g,\tau,\sigma,a)$,
 	\item[2.] If $|\tau|<1$, the system \eqref{CE} associated with $(g,\tau^a,\sigma)$ has at least two solutions for all $a$ large enough only depending on $(g,\tau,\sigma,c)$.
 \end{itemize}
 \end{Theorem}
The outline of this article is as follows. In Section 2 we introduce the concept of half-continuity and then make a generalization of classical fixed point theorems.
In Section 3 we establish some general results on the Lichnerowicz equation, that is the first equation in \eqref{CE}. In the last section we will apply our new fixed point theorem to give the proof of Theorem \ref{theorem.main}.
\section{The half-continuity method}

\subsection{Motivation}\label{motivation}
We begin by recalling Schaefer's fixed point theorem, which is a direct consequence of Leray--Schauder's fixed point theorem. For the proof, we refer the reader to \cite[Theorem 11.6]{GTelliptic}.
\begin{theorem}\label{Schaefer FPT}
	Let $X$ be a Banach space and assume that $T~:~ [0,1]\times X \longrightarrow X$ is a continuous compact operator. Set
	$$
	K:=\bigg\{(t,x)\in [0,1]\times X \quad\text{s.t.}\quad x=tT(t,x) \bigg\}.
	$$
	If $K$ is bounded, then $T(1,.)$ has a fixed point.
\end{theorem}
 Let $(X,\,T,\,K)$ be as in Theorem \ref{Schaefer FPT}. For any $c>0$ we consider the map $F_c: X \longrightarrow \mathbb{R}$ defined by 
$$
F_c(t,x):=\|x\|-c.
$$
It follows by Theorem \ref{Schaefer FPT} that $T(1,.)$ has a fixed point provided
\begin{equation}\label{motivation 2}
\bigg\{(t,x)\in K~:~F_c(x)=0\bigg\} = \emptyset
\end{equation}
for some (and hence all) $c$ large enough. 

Now let $\{F_i\}_{1\le i\le l}$ with $l\in \mathbb{N}_+$ be a certain finite sequence of real-valued functions on $[0,1]\times X$. One should think of $\{F_i\}_{1\le i\le l}$ as a ``generalization" of $\{F_c\}$. A natural question to ask is under what conditions on $\{F_i\}_{1\le i\le l}$ the map $T(1,.)$ has a fixed point as long as
$$
\bigg\{(t,x)\in K~\Big|~\prod_{i=1}^lF_i(t,x)=0~~\text{and}~~F_i(t,x)\le 0~~\text{for all $1\le i\le l$}\bigg\}= \emptyset.
$$
In this section we will give an answer to the question by using the concept of half-continuity. For the convenience of the reader, we summarize results on half-continuous maps in the next subsection. The interested reader is referred to \cite{P.Bich,TK} for more details.
\subsection{The concept of half continuity}
\begin{definition}
	Let $C$ be a subset of a Banach space $X$. A map $f~:~C \rightarrow X$ is said to be half-continuous if for each $x\in C$ with $x\ne f(x)$ there exists $p\in X^{*}$ and a neighborhood $W$ of $x$ in $C$ such that
	$$\langle p,f(y)-y\rangle >0$$
	for all $y\in W$ with $y\ne f(y)$.
\end{definition} 
The following proposition gives a relation between  half-continuity and continuity.

\begin{proposition}[see \cite{TK}, Proposition 3.2]\label{proposition cont and half cont}
	Let $X$ be a Banach space and $C$ be a subset of $X$. Then every continuous map $f~:~C\rightarrow X$ is half-continuous.
\end{proposition}

\begin{remark}[see \cite{TK}]
	There are some half-continuous maps which are not continuous. For example, let $f~:~\mathbb{R}\rightarrow \mathbb{R}$ be defined by
	$$
	f(x)=\left\{\begin{array}{ll}
	3&\textrm{if $x\in[0,1)$,}\\
	2&\textrm{otherwise.}
	\end{array} \right.
	$$
	Then $f$ is half-continuous but not continuous.
\end{remark}

\begin{theorem}[see \cite{P.Bich,TK}] \label{fixed point theorem}
	Let $C$ be a nonempty compact convex subset of a Banach space $X$. If $f~:~C\rightarrow C$ is half-continuous, then $f$ has a fixed point. 
\end{theorem}

A direct consequence of Theorem \ref{fixed point theorem} is the following corollary, which is our main tool in the next subsection.

\begin{corollary}\label{Fixed point theorem}
	Let $C$ be a nonempty closed convex subset of a Banach space $X$. If $f:C\rightarrow ~C$ is half-continuous and $f(C)$ is precompact, then $f$ has a fixed point.
\end{corollary}
\begin{proof}
	Since $\overline{f(C)}$ is nonempty compact and $X$ is a Banach space, $\overline{\text{conv}}(f(C))$ is a nonempty compact convex subset of $X$
	(see \cite{Rudin}, Theorem 3.20). Moreover, since $C$ is a closed convex subset of $X$ and $f(C)\subset C$, we have $\overline{\text{conv}}(f(C))\subset C,$
	and hence $f\left(\overline{\text{conv}}(f(C))\right)\subset f(C)\subset \overline{\text{conv}}(f(C)).$ Now restricting $f$ to $\overline{\text{conv}}(f(C))$
	and applying the previous theorem, we obtain the desired conclusion. 
\end{proof}
\subsection{A generalization of classical fixed point theorems}
We now present how to apply Corollary \ref{Fixed point theorem} to address the question posed in Subsection \ref{motivation}.  We first make the following definition. 
\begin{definition}\label{reasonable}
	Let $X$ be a Banach space and let $T~:~[0,1]\times X\longrightarrow X$ be a continuous compact operator. Let $\{F_i\}_{1\leq i\leq l}$ with $l\in \mathbb{N}_{+}$ be a finite sequence of real-valued continuous functions on $[0,1]\times X$. We call $\{F_i\}_{1\leq i\leq l}$ a \textbf{$T-$association} if
	\begin{itemize}
		\item[1.] $F_i(0,0)<0$ for all $1\le i \le l$,
		\item[2.] $\bigg\{T(t,x)~\big|~\text{$(t,x)\in[0,1]\times X$ s.t. $F_i(t,x)\leq 0\quad \forall i=\overline{1,l}$}\bigg\}$ is bounded.
	\end{itemize} 
\end{definition}
\begin{example}\label{example}
	Let $X$ be a Banach space and let $T~:~ [0,1]\times X \longrightarrow X$ be a continuous compact operator. For any constant $a>0$, we define $F~:~[0,1]\times X\longrightarrow \mathbb{R}$ by
	$$
	F(t,x):=\|x\|-a.
	$$
	The sequence $\{F\}$ is of course a $T-$association by the definition. 
\end{example}
Let $\big(X,\,T,\,\{F_i\}_{1\le i\le l}\big)$ be as in Definition \ref{reasonable}. We define $S$ from $[0,1]\times X$ into itself by
\begin{equation}\label{define S}
S(t,x)=\left\{\begin{array}{ll}
\big(1,T(t,x)\big) &\textrm{if $F_i(t,x)\leq 0$ for all $1\le i\le l$},\\
(0,0)&\textrm{otherwise.}
\end{array} \right.
\end{equation}
Since $\Big\{T(t,x)~\big|~ F_i(t,x)\leq 0~~\forall i=\overline{1,l}\Big\}$ is assumed to be bounded, we can take $C>0$ be a constant satisfying 
\begin{equation}\label{bound}
\sup\Big\{\left|\left|T(t,x)\right|\right|~\big|~ F_i(t,x)\leq 0~~\forall i=\overline{1,l}\Big\} \leq C.
\end{equation}
We define
\begin{equation}\label{define mathcal C}
\mathcal{C}:=\big\{(t,x)\in [0,1]\times X\quad\text{s.t.}\quad \|x\|\leq C\big\}.
\end{equation}
Before going further, we establish some basic properties of $S$.
\begin{claim}\label{bounded and  precompact}
	$S$ maps from $\mathcal{C}$ into itself and $S(\mathcal{C})$ is precompact.
\end{claim}
\begin{proof}
	This claim is a direct consequence of the definition of $S$ and the fact that $T$ is compact.
\end{proof}
\begin{claim}\label{FP of S and T}
	If $(t_0,x_0)$ is a fixed point of $S$, then $(t_0,x_0)=(1,T(1,x_0))$ and $F_i(1,x_0)\leq 0$ for all $1\le i\le l$.
\end{claim}
\begin{proof}
	If $(t_0,x_0)=(0,0)$, we have by our assumption that $F_i(t_0,x_0)=F_i(0,0)<0$ for all $1\le i\le l$. However, by the definition of $S$ this leads to the contradiction that
	$$
	(0,0)=(t_0,x_0)=S(t_0,x_0)=(1,T(t_0,x_0)).
	$$
	Thus, we establish
	$(t_0,x_0)\ne (0,0)$ and hence by the definition of $S$ 
	$$
	F_i(t_0,x_0)\leq 0 \quad \text{for all}\quad 1\le i\le l \quad \text{and}\quad (t_0,x_0)=S(t_0,x_0)=(1,T(t_0,x_0)).
	$$ The proof is completed. 
\end{proof}
\begin{claim}\label{at F ne 0}
	$S$ is half continuous at all $(t,x)$ satisfying $\prod_{i=1}^lF_i(t,x)\ne 0$ or $F_i(t,x)>0$ for some $i\in \{1,...,l\}$.
\end{claim}
\begin{proof}
	Since $T$ and $F_i$ are continuous,  the definition of $S$ gives us that so is $S$ at all $(t,x)$ satisfying $\prod_{i=1}^lF_i(t,x)\ne 0$ or $F_i(t,x)>0$ for some $i\in \{1,...,l\}$. Thus, the claim follows by Proposition \ref{proposition cont and half cont}. 
\end{proof} 
\begin{claim}\label{at x ne tT}
	$S$ is half continuous at all $(t,x)$ satisfying $x\ne tT(t,x)$.  
\end{claim}
\begin{proof}
	Let $(t_0,x_0)\in [0,1]\times X$ satisfy $x_0\ne t_0T(t_0,x_0)$. Since $X$ is Banach, there exists $p\in X^*$ s.t.
	\begin{equation}\label{linear continuous function}
	p(x_0)<t_0p(T(t_0,x_0)).
	\end{equation}
	If $t_0=0$, we have by \eqref{linear continuous function} that
	\begin{equation}\label{t=0, 1}
	p(x_0)<0.
	\end{equation}
	We define
	$$
	P(t,x):=kt+p(x),
	$$ 
	where $k$ is a constant satisfying
	\begin{equation}\label{t=0, 2}
	k+p(T(0,x_0))-p(x_0)>0.
	\end{equation}
	It is easy to check that $P\in \big([0,1]\times X\big)^*$. Since $T$ and $p$ are continuous, by \eqref{t=0, 1}-\eqref{t=0, 2}  there exists a neighborhood $B$ of $(0,x_0)$ s.t. for all $(t,x)\in B$
	\begin{equation}\label{x ne tT}
	\left\{\begin{array}{ll}
	k(1-t)+p(T(t,x))-p(x)>0,\\
	-tk - p(x)>0.
	\end{array} \right.
	\end{equation}
	It follows that $P(S(t,x))-P(t,x)>0$ for all $(t,x)\in B$, and hence $S$ is half continuous at $(0,x_0)$ by the definition.\\
	\\
	Now assume that $t_0\ne 0$. We define
	$$
	Q(t,x):=-ht+t_0p(x), 
	$$
	where thanks to \eqref{linear continuous function}, $h$ is a constant satisfying
	\begin{equation}\label{constant h}
	p(x_0)<h<t_0p(T(t_0,x_0)).
	\end{equation}
	We may easily check that $Q\in \big([0,1]\times X\big)^*$. Note that  by \eqref{constant h} 
	$$
	\left\{\begin{array}{ll}
	-h(1-t_0)+t_0\big(p(T(t_0,x_0))-p(x_0)\big)>0,\\
	ht_0 - t_0p(x_0)>0.
	\end{array} \right.
	$$
	Since $T$ and $p$ are continuous, it follows that there exists a neighborhood $B_1$ of $(t_0,x_0)$ s.t. for all $(t,x)\in B_1$
	\begin{equation}
	\left\{\begin{array}{ll}
	-h(1-t)+t_0\big(p(T(t,x))-p(x)\big)>0,\\
	ht - t_0p(x)>0.
	\end{array} \right.
	\end{equation}
	In other words, $Q\big(S(t,x)\big)-Q(t,x)>0$ for all $(t,x)\in B_1$, and hence $S$ is half continuous at $(t_0,x_0)$ by the definition. The proof is completed.
\end{proof}
We now state the main result of this section.
\begin{theorem}\label{theorem generalization}
	Let $X$ be a Banach space and assume that $T~:~ [0,1]\times X \longrightarrow X$ is a continuous compact operator. Assume furthermore that  $\{F_i\}_{1\le i\le l}$ with $l\in \mathbb{N}_+$ is a $T$-association. Then at least one of the following assertions is true:
	\begin{itemize}
		\item[(i)] $T(1,.)$ has a fixed point,
		\item[(ii)] $\bigg\{(t,x)\in[0,1]\times X~\big|~ x=tT(t,x),\,\prod_{i=1}^{l}F_i(t,x)=0~\text{and}~ F_i(t,x)\leq 0~ \forall i=\overline{1,l}\bigg\}\ne \emptyset$.
	\end{itemize}
\end{theorem}
\begin{remark}
	Let $(X,\,T,\,F)$ be as in Example \ref{example}. Assume that $T(1,.)$ has no fixed point. Then Theorem \ref{theorem generalization} tells us that for all $a>0$ there exists $(t_a,x_a)$ s.t. $x_a=t_aT(t_a,x_a)$ and $\|x_a\|=a$. In particular, the set 
	$$
	K=\{(t,x)\in[0,1]\times X~|~x=tT(t,x) \}
	$$ 
	is unbounded. Hence Theorem \ref{theorem generalization} is a generalization of Schaefer's fixed point theorem.
\end{remark}
\begin{proof}
	Let  $S,\,\mathcal{C}$  be defined in \eqref{define S},\,\eqref{define mathcal C} respectively. Assume that $T(1,.)$ has no fixed point. We need to verify that the second assertion is true.
	
	In fact, since $T(1,.)$ has no fixed point, we obtain by Claim \ref{FP of S and T} that neither does $S$. It follows  by Corollary \ref{Fixed point theorem} and Claim \ref{bounded and  precompact}  that $S$ is not half continuous on $\mathcal{C}$. Therefore, Claims \ref{at F ne 0} and \ref{at x ne tT} give us that there exists $(t,x)\in [0,1]\times X$ s.t. $x=tT(t,x)$, $\prod_{i=1}^{l}F_i(t,x)=0$ and $F_i(t,x)\leq 0$ for all $i=\overline{1,l}$, otherwise $S$ is half continuous which is a contradiction. The proof is completed.
\end{proof} 
\section{The Lichnerowicz equation}
In this section we will review some standard facts about the Lichnerowicz equation on a compact $n-$manifold $M$:
\begin{equation}\label{Lichnerowicz}
\frac{4(n-1)}{n-2}\Delta \varphi+R\varphi+\frac{n-1}{n}\tau^{2}\varphi^{N-1}=\frac{w^{2}}{\varphi^{N+1}},
\end{equation}
here we remind the reader that 
$$
\Delta \varphi = - g^{ij}\big(\del_i\del_j - \Gamma^k_{ij}\del_k\big)\varphi.
$$
From now on, we use standard notations for function spaces, such as $L^{p}$, $C^{k}$, and Sobolev spaces $W^{k,p}$. It will be clear from the context if the notation refers to a space of functions on $M$, or a space of sections of some bundle over $M$. For spaces of functions which embed into $L^{\infty}$, the subscript $+$ is used to indicate the cone of positive functions.
We will sometimes write, for instance, $C(\alpha_{1},\alpha_{2})$ to indicate that a constant $C$ depends only on $\alpha_{1}$ and $\alpha_{2}$.

Given a function $w$ and $p>n$, we say that $\varphi_{+}\in W_{+}^{2,p}$ is a \textit{supersolution} to (\ref{Lichnerowicz}) if
$$
\frac{4(n-1)}{n-2}\Delta \varphi_{+}+R\varphi_{+}+\frac{n-1}{n}\tau^{2}\varphi_{+}^{N-1}\geq\frac{w^{2}}{\varphi_{+}^{N+1}}.
$$
A \textit{subsolution} is defined similarly with the reverse inequality.
\begin{proposition}[see \cite{Isenberg}, \cite{MaxwellRoughCompact}] \label{lemma method sub-super}
	Assume $g\in W^{2,p}$ and $w,\tau\in L^{2p}$ for some $p>n$. If $\varphi_{-},\varphi_{+}\in W_{+}^{2,p}$ are a subsolution and a supersolution respectively to (\ref{Lichnerowicz}) associated with a fixed $w$ such that $\varphi_{-}\leq \varphi_{+}$, then there exists a solution $\varphi\in W_{+}^{2,p}$ to (\ref{Lichnerowicz}) such that $\varphi_{-}\leq \varphi\leq \varphi_{+}$.
\end{proposition}
Next let us denote by $\mathcal{Y}_g$ the Yamabe invariant of the conformal class of $g$, that is
$$
\mathcal{Y}_{g}=\inf_{\substack{f\in C^{\infty}(M)\\f\ne 0}}\frac{\frac{4(n-1)}{n-2}\int_{M}{|\nabla f|^{2}dv}+\int_{M}{Rf^{2}dv}}{\|f\|^{2}_{L^{N}(M)}}.
$$
The following theorem provides existence and uniqueness of solutions to \eqref{Lichnerowicz}.
\begin{theorem}[see \cite{MaxwellRoughCompact}]\label{theorem. maxwell1}
	Assume $w,\tau\in L^{2p}$ and $g\in W^{2,p}$ for some $p>n$. Then there exists a positive solution $\varphi\in W^{2,p}_{+}$ to (\ref{Lichnerowicz}) if and only if one of the following assertions is true.
	\begin{itemize}
		\item[1.] $\mathcal{Y}_{g}>0$ and $w\ne 0$,
		\item[2.] $\mathcal{Y}_{g}=0$ and $w\ne 0$, $\tau\ne 0$,
		\item[3.] $\mathcal{Y}_{g}<0$ and there exists $\hat{g}$ in the conformal class of $g$ such that $R_{\hat{g}}=-\frac{n-1}{n}\tau^{2}$,
		\item[4.]  $\mathcal{Y}_{g}=0$ and $w\equiv 0$, $\tau\equiv 0.$
	\end{itemize}
	In Cases $1-3$ the solution is unique. In Case $4$ any two solutions are related by a scaling by a positive constant multiple.
\end{theorem}
The next lemma plays an important role in the study of \eqref{Lichnerowicz}. It is called the conformal covariance of the Lichnerowicz equation.
\begin{lemma}[see \cite{Isenberg,MaxwellNonCMC}]\label{lemma. maxwell2}
	Assume $g\in W^{2,p}$ and $w,\tau\in L^{2p}$ for some $p>n$. Assume also that $\theta\in W^{2,p}_{+}$. Define
	$$
	\hat{g}=\theta^{N-2}g,~~~\hat{w}=\theta^{-N}w,~~~\hat{\tau}=\tau.
	$$
	Then $\varphi$ is a supersolution (resp. subsolution) to (\ref{Lichnerowicz}) if and only if $\hat{\varphi}=\theta^{-1}\varphi$ is a supersolution (resp. subsolution) to the conformally transformed equation
	\begin{equation}\label{Lichnerowicz2}
	\frac{4(n-1)}{n-2}\Delta_{\hat{g}} \hat{\varphi}+R_{\hat{g}}\hat{\varphi}+\frac{n-1}{n}\hat{\tau}^{2}\hat{\varphi}^{N-1}=\frac{\hat{w}^{2}}{\hat{\varphi}^{N+1}}.
	\end{equation}
	In particular, $\varphi$ is a solution to (\ref{Lichnerowicz}) if and only if $\hat{\varphi}$ is a solution to (\ref{Lichnerowicz2}).
\end{lemma} 
\begin{lemma}[Maximum principle]\label{lemma. maxpriw} Given $g\in W^{2,p}$ for some $p>n$, we assume that $\theta,~\varphi$ are a supersolution (resp. subsolution) and a positive solution respectively to (\ref{Lichnerowicz}) with a fixed $(\tau,\,w)$. Then 
	$$
	\theta\geq \varphi~(\mbox{resp. $\leq$}).
	$$
	Consequently, for any $w_0,\,w_1\in L^{2p}$ let $\varphi_{0},\,\varphi_1$ be solutions to \eqref{Lichnerowicz} associated with $(\tau,\,w_0),\, (\tau,\,w_1)$ respectively. Assume that $w_{1}^2\geq w_{0}^2$, then $\varphi_{1}\geq \varphi_{0}$. Moreover,  if there exists a constant $c>0$ s.t. $w_1^2 - w_0^2 \ge c$, then $\varphi_1>\varphi_0$.
\end{lemma}
\begin{proof} 
	We will prove the supersolution case. The remaining cases are similar. Assume that $\theta,\varphi$ are a supersolution and a positive solution respectively of (\ref{Lichnerowicz}) associated with a fixed $w$. Since $\varphi$ is a solution, $\varphi$ is also a subsolution, and hence, as is easily checked so is $t\varphi$ for all constant $t\in (0,1]$. Since $\min{\theta}>0$, we now take $t$ small enough s.t. $t\varphi\leq \theta$. By Proposition \ref{lemma method sub-super}, we then conclude that there exists a solution $\varphi^{\prime}\in W^{2,p}$ to $(\ref{Lichnerowicz})$ satisfying $t\varphi\leq \varphi^{\prime}\leq \theta$. On the other hand, by uniqueness of positive solutions to (\ref{Lichnerowicz}) given by Theorem \ref{theorem. maxwell1}, we obtain that $\varphi=\varphi^{\prime}$, and hence get the desired conclusion.
	
	Now let $\varphi_{0},\,\varphi_1$ be solutions to \eqref{Lichnerowicz} associated with $(\tau,\,w_0),\, (\tau,\,w_1)$ respectively. If $w_1^2\ge w_0^2$, then $\varphi_1$ is a supersolution to \eqref{Lichnerowicz} associated with $(\tau,\,w_0)$, and hence $\varphi_1\ge \varphi_0$ as we have shown above. In the case where $w_1^2 - w_0^2\ge c$ for some constant $c>0$, we define $\varphi_\epsilon=\varphi_1-\epsilon$ with $\epsilon>0$. We have that
	$$
	\aligned
	&\frac{4(n-1)}{n-2}\Delta \varphi_\epsilon+R\varphi_\epsilon+\frac{n-1}{n}\tau^{2}\varphi_\epsilon^{N-1}-\frac{w_0^{2}}{\varphi_\epsilon^{N+1}}
	\\
	=&\frac{w_1^{2}-w_0^2}{\varphi_1^{N+1}}-w_0^2\bigg(\frac{1}{\varphi_\epsilon^{N+1}}-\frac{1}{\varphi_1^{N+1}}\bigg)-R(\varphi_1-\varphi_\epsilon)-\frac{n-1}{n}\tau^{2}\big(\varphi_1^{N-1}-\varphi_\epsilon^{N-1}\big)
	\endaligned
	$$
Note that the last three terms converge to $0$ as $\epsilon\to 0$ while the first term is strictly positive by the assumption. Then we obtain that 
$$
\frac{4(n-1)}{n-2}\Delta \varphi_\epsilon+R\varphi_\epsilon+\frac{n-1}{n}\tau^{2}\varphi_\epsilon^{N-1}-\frac{w_0^{2}}{\varphi_\epsilon^{N+1}}>0
$$
as long as $\epsilon>0$ is sufficiently small. In other words, $\varphi_\epsilon$ is a supersolution to \eqref{Lichnerowicz} associated with $(\tau,\,w_0)$, and hence $\varphi_0\le \varphi_\epsilon<\varphi_1$. The proof is completed.
\end{proof}
We now restrict our discussion to the Lichnerowicz equation with positive Yamabe invariant. Let $g\in W^{2,p}$ with $p>n$ be a Yamabe-positive metric on $M$. Given $\tau\in L^{2p}$ we consider the map
$\mathcal{L}:\big(L^{\infty}\setminus \{0\}\big)\times [0,1]\rightarrow W_{+}^{2,p}$ defined by
$\mathcal{L}(w,t)=\theta$, where
$$
\frac{4(n-1)}{n-2}\Delta \theta+R\theta=-\frac{n-1}{n}t\tau^{2}\theta^{N-1}+w^{2}\theta^{-N-1}.
$$
Note that $\mathcal{L}$ is well-defined by Case 1 in Theorem \ref{theorem. maxwell1}. The following result ensures continuity of $\mathcal{L}$.
\begin{proposition}[see \cite{MaxwellNonCMC}, \cite{Nguyen}]\label{proposition L continuous}
	$\mathcal{L}$ is a $C^1$-map.
\end{proposition}
\begin{proof}
Since $\mathcal{Y}_g>0$ and since the Lichnerowicz equation is conformally covariant as shown in Lemma \ref{lemma. maxwell2}, we may assume without loss of generality that $R>0$. 

We will prove the proposition by the implicit function theorem. In fact, we define
$$
\aligned
F:(L^{\infty} \setminus \{0\})\times [0,1]\times W_{+}^{2,p}&\longrightarrow L^{2p}
\\
(w,t,\theta)\qquad &\longmapsto \frac{4(n-1)}{n-2}\Delta \theta+R\theta+\frac{n-1}{n}t\tau^{2}\theta^{N-1}-w^{2}\theta^{-N-1}.
\endaligned
$$
It is clear that $F$ is continuous and $F(w,t,\mathcal{L}(w,t))=0$ for all $(w,t)\in (L^{\infty} \setminus \{0\})\times [0,1]$.
A standard computation shows that the Fr\'{e}chet derivative of $F$ w.r.t. $\theta$ is given by
$$F_{\theta}(w,t)(u)=\frac{4(n-1)}{n-2}\Delta u+Ru+\frac{(N-1)(n-1)}{n}t\tau^{2}\theta^{N-2}u+(N+1)w^{2}\theta^{-N-2}u.
$$
We first note that $F_{\theta}\in C\big((L^{\infty} \setminus \{0\})\times [0,1], L(W^{2,p},L^p)\big)$, where $L(W^{2,p}, L^p)$ denotes the Banach space of all linear continuous maps from $W^{2,p}$ into $L^p$.
Now, given $(w_{0},t_{0})\in (L^{\infty} \setminus \{0\})\times [0,1]$, setting $\theta_{0}=\mathcal{L}(w_{0},t_{0})$, we have
$$
F_{\theta_{0}}(w_{0},t_{0})(u)=\frac{4(n-1)}{n-2}\Delta u+\left(R+\frac{(N-1)(n-1)}{n}t_{0}\tau^{2}\theta_{0}^{N-2}+(N+1)w_{0}^{2}\theta_{0}^{-N-2}\right)u.
$$
Since 
$$R+\frac{(N-1)(n-1)}{n}t_{0}\tau^{2}\theta_{0}^{N-2}+(N+1)w_{0}^{2}\theta_{0}^{-N-2}\geq \min R>0,$$
we conclude that $F_{\theta_{0}}(w_{0},t_{0}):~W^{2,p}\to L^p$ is an isomorphism. Therefore, the implicit function theorem implies that $\mathcal{L}$ is a $C^{1}$ function in a
neighborhood of $(w_{0},t_{0})$, which completes the proof.
\end{proof}
\section{The proof of Theorem \ref{theorem.main}}
This section is devoted to the proof of Theorem \ref{theorem.main}. As we will see in the proof, one of the two solutions in our nonuniqueness result is indeed a ``small TT-tensor" solution shown in \cite{HNT2, MaxwellNonCMC, GicquaudNgo, Nguyen}. So for the convenience of the reader, we begin by recalling what this solution means.
\begin{theorem}[see \cite{Nguyen}, Theorem 4.8 and Remark 4.9]\label{theorem.farCMC}
Let $g\in W^{2,p}$ with $p > n$ be a Yamabe-positive metric on a smooth compact $n-$manifold. Assume that $g$ has no conformal Killing vector field, $\sigma \in W^{1,p}\setminus \{0\}$ and $\tau\in W^{1,p}$. There exist $\epsilon=\epsilon(g)>0$ and a constant $c_1=c_1(g,\sigma)>0$ s.t. as long as
$$
\|d\tau\|_{L^p}^{N+2}\|\sigma\|_{L^2}^{N-2}\le \epsilon,
$$
the system \eqref{CE} admits a solution $(\varphi,\,W)$ satisfying
$$
\|\varphi\|_{L^\infty}\le c_1. 
$$
\end{theorem} 

We now consider our main result. Because the arguments for the two assertions in Theorem \ref{theorem.main} are broadly similar, in which the second one is more complicated, we will solve it first and then sketch the proof of the first one.
\begin{theorem}\label{theorem-nonuniqueness 1}
Let $g\in W^{2,p}$ with $p > n$ be a Yamabe-positive metric on a smooth compact $n-$manifold. Assume that $g$ has no conformal Killing vector field, $\sigma \in W^{1,p}\setminus \{0\}$ and $\tau\in W^{1,p}$ does not change sign. Assume furthermore that $\tau$ satisfies 
\begin{equation}\label{key condition}
\bigg|L\bigg(\frac{d\tau}{\tau}\bigg)\bigg|\le c\bigg|\frac{d\tau}{\tau}\bigg|^2
\end{equation}
for some constant $c>0$. If $|\tau| <1$, then the system \eqref{CE} associated with $(g,\tau^a,\sigma)$  has at least two solutions for all $a>0$ large enough only depending on $(g,\tau,\sigma,c)$.
\end{theorem}
The following lemma is the key to solving the theorem.
\begin{lemma}\label{lemma - stability 1}
Let $g\in W^{2,p}$ with $p > n$ be a Yamabe-positive metric on a smooth compact $n-$manifold. Assume that $g$ has no conformal Killing vector field, $\sigma \in W^{1,p}\setminus \{0\}$ and $\tau\in W^{1,p}$ does not change sign. For $a,\,k\ge 0$ let us denote by $\mathcal{A}_a(k)$ the set of all $(\varphi,W)$ satisfying the following $(a,k)-$conformal equations
\begin{equation}\label{k-CE}
 \aligned
 \frac{4(n-1)}{n-2}\Delta \varphi+R\varphi=&-\frac{n-1}{n}\tau^{2a}\varphi^{N-1}+\left(|\sigma+LW|^{2}+k^{2}\right)\varphi^{-N-1},\\
 -\frac{1}{2}L^{*}LW=&\frac{n-1}{n}\varphi^{N}d\tau^a.
\endaligned
 \end{equation}
If $\tau$ satisfies 
$$
\bigg|L\bigg(\frac{d\tau}{\tau}\bigg)\bigg|\le c\bigg|\frac{d\tau}{\tau}\bigg|^2
$$
 for some constant $c>0$,  as long as $\mathcal{A}_a(k)\ne \emptyset$, we then have that
\begin{equation}\label{def. A(a)}
A(a):=\sup_{k\ge 0}\bigg\{\sup_{(\varphi,W)\in \mathcal{A}_a(k)}\|\varphi\|_{L^\infty}\bigg\}<+\infty
\end{equation}
for all $a>\frac{c}{2}\sqrt{\frac{n}{n-1}}$.
\end{lemma}
\begin{proof}
We argue by contradiction. Assume that the lemma is not true. Then there exists $(\varphi_i,W_i,k_i)$ with $\|\varphi_i\|_{L^{\infty}}\longrightarrow +\infty$ s.t. 
\begin{equation}\label{k-CEi}
 \aligned
 \frac{4(n-1)}{n-2}\Delta \varphi_i+R\varphi_i=&-\frac{n-1}{n}\tau^{2a}\varphi_i^{N-1}+\left(|\sigma+LW_i|^{2}+k_i^{2}\right)\varphi_i^{-N-1},\\
 -\frac{1}{2}L^{*}LW_i=&\frac{n-1}{n}\varphi_i^{N}d\tau^a.
\endaligned
 \end{equation}
 Setting 
 $$
 \gamma_i:=\max\bigg\{\|\varphi_i\|_{L^\infty},\,k_i^{1/N}\bigg\},
 $$
 we rescale $\varphi_i,\,W_i,\,\sigma$ and $k_i$ as follows:
 $$
 \widehat{\varphi}_i:=\gamma_i^{-1}\varphi_i,\qquad \widehat{W}_i:=\gamma_i^{-N}W_i,\qquad \widehat{\sigma}:=\gamma_i^{-N}\sigma, \qquad \widehat{k}_i:=\gamma_i^{-N}k_i.
 $$
The system \eqref{k-CEi} can be rewritten as
$$
\aligned
\gamma_i^{\frac{1}{N-2}}\bigg(\frac{4(n-1)}{n-2}\Delta \widehat{\varphi}_i+R\widehat{\varphi}_i\bigg)=&-\frac{n-1}{n}\tau^{2a}\widehat{\varphi}_i^{N-1}+\left(|\widehat{\sigma}+L\widehat{W}_i|^{2}+\widehat{k}_i^{2}\right)\widehat{\varphi}_i^{-N-1},
\\
-\frac{1}{2}L^{*}L\widehat{W}_i=&\frac{n-1}{n}\widehat{\varphi}_i^{N}d\tau^a.
\endaligned
$$
Letting $i\to +\infty$,  in analogy with the scaling (blow-up) arguments in \cite[Theorem 1.1]{DahlGicquaudHumbert} or \cite[Theorem 3.3]{Nguyen}  we have that (after passing to a subsequence)
$$
(\widehat{\varphi}_i^N,\,\widehat{W}_i,\,\widehat{k}_i)\to \bigg(\sqrt{\frac{n-1}{n}}\frac{|LW_\infty|+k_\infty}{\tau^a},\,W_\infty,\,k_\infty\bigg) \quad \text{in} \quad L^\infty
$$
for some $(W_\infty,\,k_\infty)\in \big(W^{2,p}\times [0,1]\big)\setminus \{(0,\,0)\}$. Hence we obtain by the vector equation that
$$
\aligned
-\frac{1}{2}L^*LW_\infty=&\sqrt{\frac{n-1}{n}}\big(|LW_\infty|+k_\infty\big)\frac{d\tau^a}{\tau^a}
\\
=&a\sqrt{\frac{n-1}{n}}\big(|LW_\infty|+k_\infty\big)\frac{d\tau}{\tau}.
\endaligned
$$
Now take the scalar product of this equation with $d\tau/\tau$ and integrate. It follows that 
\begin{equation}\label{ineq:3.1}
\aligned
a\sqrt{\frac{n-1}{n}}\int_M\big(|LW_\infty|+k_\infty\big)\bigg(\frac{d\tau}{\tau}\bigg)^2dv=&-\frac{1}{2}\int_M \bigg\langle LW_\infty,\,L\bigg(\frac{d\tau}{\tau}\bigg)\bigg\rangle dv \\
\leq& \frac{1}{2}\int_M|LW_\infty|\bigg|L\bigg(\frac{d\tau}{\tau}\bigg)\bigg|dv\\
\leq& \frac{c}{2}\int_M|LW_\infty|\bigg(\frac{d\tau}{\tau}\bigg)^2dv.
\endaligned
\end{equation}
Since
$$
a> \frac{c}{2}\sqrt{\frac{n}{n-1}},
$$
 we get from \eqref{ineq:3.1} that $|LW_\infty|+k_\infty=0$, which is a contradiction. The proof is completed.
\end{proof}
\begin{proof}[Proof of Theorem \ref{theorem-nonuniqueness 1}]
  Since $\|\tau\|_{L^\infty}<1$, we have that
  $$
  \|d\tau^a\|_{L^p}\le a\|\tau\|_{L^\infty}^{a-1}\|d\tau\|_{L^p}\to 0\quad \text{as}\quad a\to+\infty.
  $$ 
  It follows by Theorem \ref{theorem.farCMC} that there exists a sufficiently large constant $a_0=a_0(g,\tau,\sigma)>0$  s.t. for all $a\ge a_0$ the system \eqref{CE} associated with $(g,\tau^a,\sigma)$ has a solution $(\widetilde{\varphi}_a,\,\widetilde{W}_a)$ satisfying
$$
\|\widetilde{\varphi}_a\|_{L^{\infty}}\leq c_1
$$
for some $c_1=c_1(g,\sigma)>0$. Therefore, to prove the theorem, we only need to show that for all sufficiently small $t>0$ the system \eqref{CE} associated with $(g,\tau^{a_0/t},\sigma)$ admits a solution $(\varphi_t,\,W_t)$ satisfying $\|\varphi_t\|_{L^{\infty}}>c_1$. The proof will be divided into three steps.

\

\textit{Step 1. Critical elements.} Let $(\varphi_i,\,W_i,\,k_i)\in W^{2,p}_+\times W^{2,p}\times [0,+\infty)$ satisfy $(\varphi_i,\,W_i)\in \mathcal{A}_{a_0}(k_i)$ and
\begin{equation}\label{eq3:1}
\|\varphi_i\|_{L^\infty}\to A(a_0). 
\end{equation}
Here $\mathcal{A}_{a_0}(k_i)$ and $A(a_0)$ are defined in Lemma \ref{lemma - stability 1} with respect to $(a,k)=(a_0,\,k_i)$. Note that
$$
\int_{M}R\varphi_i \,d\nu + {n-1 \over n} \int_M\tau^{a_0}\varphi_i^{N-1} \,d\nu = \int_{M} \big(|\sigma + LW_i|^2 + k_i^2\big)\varphi_i^{-N-1}\,d\nu,
$$
then
$$
\|\varphi_i\|_{L^\infty}^{N+2}\int_{M}|R| \,d\nu + \|\varphi_i\|_{L^\infty}^{2N}\bigg({n-1 \over n} \int_M\tau^{a_0} \,d\nu \bigg)\ge \int_{M} \big(|\sigma + LW_i|^2 + k_i^2\big)\,d\nu.
$$
Since $A(a_0)<+\infty$, it follows that $k_i$ is bounded and hence (after passing to a subsequence) 
\bel{k converge}
k_i\to k_0.
\ee
 On the other hand, we have by the vector equation that
$$
\aligned
\|W_i\|_{W^{2,p}}\le& c_2(g)\|\varphi_i^Nd\tau^{a_0}\|_{L^p}
\\
\le& c_2\|\varphi_i\|^N_{L^\infty}\|d\tau^{a_0}\|_{L^p}
\\
\le& c_3(c_2,\tau)\big(A(a_0)\big)^N\qquad \text{(by \eqref{eq3:1})}.
\endaligned
$$
Thanks to the Sobolev embedding theorem, this gives us that (after passing to a subsequence) $W_i$ converges to $W_0$ in $C^1$. Combined with \eqref{k converge}, since
$$
\int_M |\sigma + LW_{t,\varphi}|^2 \,d\nu = \int_M |\sigma|^2 \,d\nu + \int_M |LW_{t,\varphi}|^2 \,d\nu \ge \int_M |\sigma|^2 \,d\nu>0, 
$$
we obtain by Proposition \ref{proposition L continuous} that 
$$
\varphi_i\to \varphi_0 \quad \text{in $W^{2,p}$},
$$ where $\varphi_0$ is a unique solution to the Lichnerowicz equation
\begin{equation}\label{def varphi_0}
\frac{4(n-1)}{n-2}\Delta \varphi_0+R\varphi_0=-\frac{n-1}{n}\tau^{2a_0}\varphi^{N-1}_0+\left(|\sigma+LW_0|^{2}+k_0^{2}\right)\varphi^{-N-1}_0.
\end{equation}
In particular,
$$
(\varphi_0,\,W_0)\in \mathcal{A}_{a_0}(k_0)\quad \text{and}\quad\|\varphi_0\|_{L^\infty}=\lim\|\varphi_i\|_{L^\infty}=A(a_0).
$$
The triple $(\varphi_0,\,W_0,\,k_0)$ plays an important role in our arguments as we will see in the following steps.

\

\textit{Step 2. Constructing a continuous and compact operator.} We define an operator $T:[0,1]\times C_+^0 \rightarrow C_+^0$ as follows. For each $(t,\varphi)\in (0,1]\times C_+^0$, there exists a unique $W_{t,\varphi}\in W^{2,p}$ s.t.
$$
-\frac{1}{2}L^*LW_{t,\varphi}=\frac{n-1}{n}t^{-N}\varphi^Nd\tau^{a_0/t},
$$ 
and then there is a unique $\psi_{t,\varphi}\in W_{+}^{2,p}$ s.t.
$$
\aligned
&\frac{4(n-1)}{n-2}\Delta \psi_{t,\varphi} +R\psi_{t,\varphi}+\frac{n-1}{n}\tau^{2a_0/t}\psi_{t,\varphi}^{N-1}\\
=&\bigg[|\sigma+LW_{t,\varphi}|^2+\bigg(2\max\big\{\|\varphi_0\|_{L^\infty},\,2\big\}-\|\varphi\|_{L^\infty}\bigg)_+\bigg(\|\sigma+LW_0\|^2_{L^\infty}+k_0^2\bigg)\bigg]\psi_{t,\varphi}^{-N-1}.
\endaligned
$$
Here and subsequently, for any function $f$ we write $f_+:=\max\{f,\,0\}$.
We define
$$
	T(t,\varphi):=\left\{\begin{array}{ll}
	\psi_{t,\varphi}&\textrm{if $t\ne 0$,}\\
	\psi_{0,\varphi}&\textrm{if $t=0$,}
	\end{array} \right.
$$
where $\psi_{0,\varphi}$ is a unique solution of the equation
$$
\frac{4(n-1)}{n-2}\Delta \psi +R\psi
=\bigg[|\sigma|^2+\bigg(2\max\big\{\|\varphi_0\|_{L^\infty},\,2\big\}-\|\varphi\|_{L^\infty}\bigg)_+\bigg(\|\sigma+LW_0\|^2_{L^\infty}+k_0^2\bigg)\bigg]\psi^{-N-1}.
$$
Note that $\psi_{0,\varphi}$ is well-defined by Case 1 in Theorem \ref{theorem. maxwell1} since we assumed $\mathcal{Y}_g>0$. 

It is clear that $T(t,\varphi)>0$ for all $(t,\varphi)\in [0,1]\times C^0_{+}$. Analysis similar to that in \cite{DahlGicquaudHumbert}, \cite{MaxwellNonCMC} shows that $T$ is continuous compact in $[t_0,1]\times C_+^0$ with all $t_0>0$. Now for any sequence  $\{(t_i,\varphi_i)\}\subset [0,1]\times C_+^0$ satisfying $(t_i,\,\varphi_i)\to (0,\varphi_\infty)$, since $|\tau|<1$, we have
\begin{equation}\label{t=0 1}
\|\tau\|_{L^\infty}^{2a_0/t_i}\to 0
\end{equation}
and by the vector equation
\begin{equation}\label{t=0 2}
\aligned
\|W_{t_i,\varphi_i}\|_{W^{2,p}}\le &c_2t_i^{-N}\big\|\varphi_i^{N}d\tau^{a_0/t_i}\big\|_{L^p} 
\\
\le &a_0t_i^{-N-1}\big\|\tau\|_{L^\infty}^{\frac{a_0-t_i}{t_i}}\|\varphi_i^Nd\tau\big\|_{L^p}\to 0.
\endaligned
\end{equation}
Therefore, analysis similar to the proof of Proposition \ref{proposition L continuous} shows that 
$$ 
T(t_i,\varphi_i)\to \psi_{0,\varphi_\infty}~~\text{in $W^{2,p}$},
$$
where $\psi_{0,\varphi_\infty}$ is a unique solution to the equation
$$
\frac{4(n-1)}{n-2}\Delta \psi +R\psi
=\bigg[|\sigma|^2+\bigg(2\max\big\{\|\varphi_0\|_{L^\infty},\,2\big\}-\|\varphi_\infty\|_{L^\infty}\bigg)_+\bigg(\|\sigma+LW_0\|^2_{L^\infty}+k_0^2\bigg)\bigg]\psi^{-N-1}.
$$ 
Hence, $T(t_i,\,\varphi_i)\to T(0,\varphi_\infty)$ by the definition. In the case where $t_i\to 0$ and $\{\varphi_i\}$ is bounded, the facts \eqref{t=0 1}-\eqref{t=0 2} are still true. Therefore, after passing to subsequence $\|\varphi_i\|_{L^\infty}\to L$, we also have that 
$$ 
T(t_i,\varphi_i)\to \psi_{0,L}~~\text{in $W^{2,p}$},
$$
where $\psi_{0,L}$ is a unique solution of the equation
$$
\frac{4(n-1)}{n-2}\Delta \psi +R\psi
=\bigg[|\sigma|^2+\bigg(2\max\big\{\|\varphi_0\|_{L^\infty},\,2\big\}-L\bigg)_+\bigg(\|\sigma+LW_0\|^2_{L^\infty}+k_0^2\bigg)\bigg]\psi^{-N-1}.
$$
This means that as long as $t_i\to 0$ and $\{\varphi_i\}$ is bounded, $\{T(t_i,\,\varphi_i)\}$ has a convergent subsequence. Thus we can conclude that $T$ is continuous compact in $[0,1]\times C^0_+$.

\textit{Step 3. Using the half-continuity method.}
First we will show that $T(1,.)$ has no fixed point. We argue by contradiction. Assume that $\varphi_*$ is a fixed point of $T(1,.)$, that is 
\begin{equation}\label{contradiction-eq 2}
\aligned
\frac{4(n-1)}{n-2}\Delta \varphi_* +R\varphi_*=&-\frac{n-1}{n}\tau^{2a_0}\varphi_*^{N-1}
\\
&+\bigg[|\sigma+LW_*|^2+\bigg(2\max\big\{\|\varphi_0\|_{L^\infty},\,2\big\}-\|\varphi_*\|_{L^\infty}\bigg)_+\bigg(\|\sigma+LW_0\|^2_{L^\infty}+k_0^2\bigg)\bigg]\varphi_*^{-N-1},
\\
-\frac{1}{2}L^*LW_*=&\frac{n-1}{n}\varphi_*^Nd\tau^{a_0}.
\endaligned
\end{equation}
Note that 
\begin{equation}\label{contribution of added term}
\|\varphi_*\|_{L^{\infty}}\leq A(a_0)= \|\varphi_0\|_{L^\infty},
\end{equation} 
then
$$
\aligned
\bigg(2\max\big\{\|\varphi_0\|_{L^\infty},\,2\big\}-\|\varphi_*\|_{L^\infty}\bigg)_+\bigg(\|\sigma+LW_0\|^2_{L^\infty}+k_0^2\bigg)&\ge 2\bigg(\|\sigma+LW_0\|^2_{L^\infty}+k_0^2\bigg)
\\
&\ge {2 \over \text{Vol}_g(M)}\int_M|\sigma + LW_0|^2 d\nu
\\
&\ge {2 \over \text{Vol}_g(M)} \int_M|\sigma|^2 d\nu.
\endaligned
$$
Therefore, thanks to Lemma \ref{lemma. maxpriw}, we have by \eqref{def varphi_0} and \eqref{contradiction-eq 2} that
$$
\|\varphi_*\|_{L^\infty}>\|\varphi_0\|_{L^\infty},
$$
which contradicts \eqref{contribution of added term} as claimed.

Now for any constant $\kappa$ satisfying
\begin{equation}\label{kappa bound blow}
\kappa\geq c_1+\sup\bigg\{\|T(t,\varphi)\|_{L^\infty} ~\bigg|~\|\varphi\|_{L^\infty}\leq 2\max\big\{\|\varphi_0\|_{L^\infty},\,2\big\}\bigg\},
\end{equation}
let $F_{\kappa,1},F_{\kappa,2}~:~[0,1]\times C_+^0\longrightarrow \mathbb{R}$ be defined by
$$
\aligned
F_{\kappa,1}(t,\varphi):=&\frac{\|T(t,\varphi)\|_{L^\infty}}{\max\big\{\|\varphi\|_{L^\infty},\,1\big\}}-\kappa,\\
F_{\kappa,2}(t,\varphi):=&\|\varphi\|_{L^\infty}-b_\kappa,
\endaligned
$$ 
where 
$$
b_\kappa:= \kappa+2\max\big\{\|\varphi_0\|_{L^\infty},\,2\big\}+\sup\big\{A(a)~\big|~a_0\le a \le \kappa a_0\big\}.
$$
Here we recall that $A(a)$ is defined in Lemma \ref{lemma - stability 1}. It is clear that $\{F_{\kappa,1},\,F_{\kappa,2}\}$ is a $T-$association. Since $T(1,.)$ has no fixed point, we have by Theorem \ref{theorem generalization} that there exists $(t,\varphi)$ s.t.
\begin{align}
&\varphi=tT(t,\varphi)\label{varphi=tT 2},\\
&F_{\kappa,1}(t,\varphi),\,F_{\kappa,2}(t,\varphi)\leq0,\label{F<0 2}\\
&\big(F_{\kappa,1}F_{\kappa,2}\big)(t,\varphi)=0 \label{F=0 2}.
\end{align}
By \eqref{varphi=tT 2}, \eqref{F=0 2} and the definition of $\kappa$, we must have $t\ne 0$. Setting 
$$
\psi:=T(t,\varphi),
$$
the identity \eqref{varphi=tT 2} is rewritten as
{\small \begin{equation}\label{contradiction-eq 2.1}
\aligned
\frac{4(n-1)}{n-2}\Delta \psi +R\psi=&-\frac{n-1}{n}\tau^{2a_0/t}\psi^{N-1}\\
&+\bigg[|\sigma+LW|^2+\bigg(2\max\big\{\|\varphi_0\|_{L^\infty},\,2\big\}-\|\varphi\|_{L^\infty}\bigg)_+\bigg(\|\sigma+LW_0\|^2_{L^\infty}+k_0^2\bigg)\bigg]\psi^{-N-1},\\
-\frac{1}{2}L^*LW=&\frac{n-1}{n}\psi^Nd\tau^{a_0/t}.
\endaligned
\end{equation}}
Therefore, we have by the definition of $A(a)$ that
\begin{equation}\label{a/t}
\|\psi\|_{L^\infty}\leq A\bigg(\frac{a_0}{t}\bigg).
\end{equation}
Next we will prove that $F_{\kappa,2}(t,\varphi)\ne0$. In fact, if $F_{\kappa,2}(t,\varphi)=0$, i.e., $\|\varphi\|_{L^\infty}=b_\kappa>1$, then  we get by \eqref{varphi=tT 2} 
$$
\frac{a_0}{t}=\bigg(\frac{\|\psi\|_{L^\infty}}{\|\varphi\|_{L^\infty}}\bigg)a_0=\bigg(\frac{\|\psi\|_{L^\infty}}{\max\big\{\|\varphi\|_{L^\infty},\,1\big\}}\bigg)a_0.
$$
Combined with $F_{\kappa,1}(1,\varphi)\le 0$, that is $\frac{\|\psi\|_{L^\infty}}{\max\{\|\varphi\|_{L^\infty},1\}}\le \kappa$, we obtain
$$
\frac{a_0}{t}\leq \kappa a_0.
$$
However, by \eqref{a/t} and the definition of $b_k$, this leads to the contradiction that
 $$
 b_\kappa=\|\varphi\|_{L^\infty}\leq \|\psi\|_{L^\infty}\leq A\bigg(\frac{a_0}{t}\bigg)<b_\kappa.
 $$
 Therefore, $F_{\kappa,2}\ne 0$ as claimed and hence we deduce from \eqref{F=0 2} that
 \begin{equation}\label{t=kappa}
 F_{\kappa,1}(t,\varphi)=\frac{\|\psi\|_{L^\infty}}{\max\big\{\|\varphi\|_{L^\infty},\,1\big\}}-\kappa=0.
 \end{equation}
 In particular,  we have
 $$
 \|\psi\|_{L^\infty}\geq \kappa.
 $$
 If 
 $$
 \|\varphi\|_{L^\infty} \leq 2\max\big\{\|\varphi_0\|_{L^\infty},\,2\big\},
 $$
 it follows from the property \eqref{kappa bound blow} of $\kappa$ that 
 $$
 \|\psi\|_{L^\infty}<\kappa,
 $$
 which is a contradiction. Therefore, we must have  
 \begin{equation}\label{varphi > 2max}
 \|\varphi\|_{L^\infty} > 2\max\big\{\|\varphi_0\|_{L^\infty},\,2\big\}>1,
 \end{equation}
  and hence by \eqref{varphi=tT 2} and \eqref{t=kappa}
 \begin{equation}\label{1/t=kappa}
 \frac{1}{t}=\frac{\|\psi\|_{L^\infty}}{\|\varphi\|_{L^\infty}}=\kappa.
 \end{equation}
 Taking \eqref{varphi > 2max}-\eqref{1/t=kappa} into \eqref{contradiction-eq 2.1}, we obtain
 $$
 \aligned
 \frac{4(n-1)}{n-2}\Delta \psi +R\psi=&-\frac{n-1}{n}\tau^{2\kappa a_0}\psi^{N-1}+|\sigma+LW|^2\psi^{-N-1},\\
 -\frac{1}{2}L^*LW=&\frac{n-1}{n}\psi^Nd\tau^{\kappa a_0}.
 \endaligned
 $$
Since $\|\psi\|_{L^{\infty}}>\kappa>c_1$ and since $\kappa$ is an arbitrary constant satisfying \eqref{kappa bound blow}, the theorem follows.
\end{proof}
As the reader may have noticed, the key aspects of Theorem \ref{theorem-nonuniqueness 1} are smallness of $\tau^a$ and nonexistence of non-zero solutions $W$ to the limit equations
$$
-\frac{1}{2}L^*LW=\sqrt{\frac{n-1}{n}}|LW|\frac{d\tau^a}{\tau^a}.
$$
In this sense, as we will see below, the first assertion in Theorem \ref{theorem.main} can be understood to be another ``version" of the second one. 
\begin{theorem}
Let $g\in W^{2,p}$ with $p > n$ be a Yamabe-positive metric on a smooth compact $n-$manifold. Assume that $g$ has no conformal Killing vector field, $\sigma \in W^{1,p}\setminus \{0\}$ and $\tau\in W^{1,p}$ does not change sign. Assume furthermore that 
$$
\bigg|L\bigg(\frac{d\tau}{\tau}\bigg)\bigg|\le c\bigg|\frac{d\tau}{\tau}\bigg|^2
$$
for some constant $c>0$. Given $a>\frac{c}{2}\sqrt{\frac{n}{n-1}}$ the system \eqref{CE} associated with $(g,t\tau^a,\sigma)$  has at least two solutions for all $t$ small enough only depending on $(g,\tau,\,\sigma,a)$.
\end{theorem}
\begin{proof}
	We have by Theorem \ref{theorem.farCMC} that there exists $t_0>0$ small enough only depending on $(g,\,\tau,\,\sigma,\,a)$ s.t. for all $t\le t_0$ the system \eqref{CE} associated with $(g,t\tau^a,\sigma)$ has a solution $(\widetilde{\varphi}_t,\,\widetilde{W}_t)$ satisfying
	$$
	\|\widetilde{\varphi}_t\|_{L^\infty}\le c_1
	$$ 
for some $c_1=c_1(g,\sigma)>0$.	Therefore, to prove the theorem, it suffices to show that \eqref{CE} associated with $(g,\,tt_0\tau^a,\sigma)$ admits a solution $(\varphi_t,\,W_t)$ satisfying  $\|\varphi_t\|_{L^\infty}>c_1$ for all sufficiently small $t>0$.
	
In fact, for any $t,k>0$ let us denote by $\mathcal{A}_t(k)$ the set of all $(\varphi,\,W)$ satisfying the following $(t,k)-$ conformal equations
	\begin{equation}\label{t-CE}
	\aligned
	\frac{4(n-1)}{n-2}\Delta \varphi+R\varphi=&-\frac{n-1}{n}t^2\tau^{2a}\varphi^{N-1}+\left(|\sigma+LW|^{2}+k^{2}\right)\varphi^{-N-1},\\
	-\frac{1}{2}L^{*}LW=&\frac{n-1}{n}t\varphi^{N}d\tau^a.
	\endaligned
	\end{equation}
	Since $a>\frac{c}{2}\sqrt{\frac{n}{n-1}}$, in analogy with Lemma \ref{lemma - stability 1} we have that for any $t>0$
	\begin{equation}\label{def. A(t)}
	A(t):=\sup_{k>0}\bigg\{\sup_{(\varphi,W)\in \mathcal{A}_t(k)}\|\varphi\|_{L^\infty}\bigg\}<+\infty.
	\end{equation}
	Next similarly to Step 1 in the proof of Theorem  \ref{theorem-nonuniqueness 1}, we may take $(\varphi_{0},\,W_{0},\,k_0)$ s.t. $(\varphi_0,\,W_0)$ satisfies the $(t_0,\,k_0)-$conformal equations \eqref{t-CE} and
	$$
	\|\varphi_{0}\|_{L^\infty}=\mathcal{A}(t_0).
	$$
	Now we define a continuous and compact operator $T~:~[0,1]\times C_+^0\rightarrow C_+^0$ as follows. For each $\varphi\in C_+^0$, 
	 there exists a unique $W_{\varphi}\in W^{2,p}$ s.t.
	$$
	-\frac{1}{2}L^*LW_{\varphi}=\frac{n-1}{n}t_0\varphi^Nd\tau^{a},
	$$ 
	and there is a unique $\psi_{t,\varphi}\in W_{+}^{2,p}$ s.t.
	$$
	\aligned
	&\frac{4(n-1)}{n-2}\Delta \psi_{t,\varphi} +R\psi_{t,\varphi}+\frac{n-1}{n}t^{2N}t_0^2\tau^{2a}\psi_{t,\varphi}^{N-1}\\
	=&\bigg[|\sigma+LW_{\varphi}|^2+\bigg(2\max\big\{\|\varphi_{0}\|_{L^\infty},\,2\big\}-\|\varphi\|_{L^\infty}\bigg)_+\bigg(\|\sigma+LW_0\|^2_{L^\infty}+k_0^2\bigg)\bigg]\psi_{t,\varphi}^{-N-1}.
	\endaligned
	$$
We define
$$
T(t,\varphi):=\psi_{t,\varphi}.
$$
In view of Proposition \ref{proposition L continuous}, it follows by \cite{DahlGicquaudHumbert, MaxwellNonCMC}  that $T$ is a continuous and compact operator. Moreover, analysis similar to that in the proof of Theorem \ref{theorem-nonuniqueness 1} shows that $T(1,.)$ has no fixed point.

Next for an arbitrary $\kappa$ satisfying
\begin{equation}\label{2.kappa bound blow}
\kappa\geq c_1+\sup\bigg\{\|T(t,\varphi)\|_{L^\infty} ~\bigg|~\|\varphi\|_{L^\infty}\leq 2\max\big\{\|\varphi_0\|_{L^\infty},\,2\big\}\bigg\},
\end{equation}
let $F_{\kappa,1},F_{\kappa,2}~:~[0,1]\times C_+^0\longrightarrow \mathbb{R}$ be defined by
$$
\aligned
F_{\kappa,1}(t,\varphi):=&\frac{\|T(t,\varphi)\|_{L^\infty}}{\max\big\{\|\varphi\|_{L^\infty},\,1\big\}}-\kappa,\\
F_{\kappa,2}(t,\varphi):=&\|\varphi\|_{L^\infty}-b_\kappa,
\endaligned
$$ 
where 
$$
b_\kappa:= \kappa+2\max\big\{\|\varphi_0\|_{L^\infty},\,2\big\}+\sup\big\{A(t)~\big|~\kappa^{-N}t_0\le t \le t_0\big\}.
$$
We may easily check that $\{F_{\kappa,\,1},\,F_{\kappa,\,2}\}$ is a $T-$association. Since $T(1,.)$ has no fixed point, we have by Theorem \ref{theorem generalization} that there exists $(t,\varphi)$ s.t.
\begin{align}
&\varphi=tT(t,\varphi)\label{2. varphi=tT 2},\\
&F_{\kappa,1}(t,\varphi),\,F_{\kappa,2}(t,\varphi)\leq0,\label{2. F<0 2}\\
&\big(F_{\kappa,1}F_{\kappa,2}\big)(t,\varphi)=0 \label{2. F=0 2}.
\end{align}
By \eqref{2. varphi=tT 2}, \eqref{2. F=0 2} and the definition of $\kappa$, we have $t\ne 0$. Setting 
$$
\psi:=T(t,\varphi),
$$
the identity \eqref{2. varphi=tT 2} is rewritten as
{\small \begin{equation}\label{2. contradiction-eq 2.1}
	\aligned
	\frac{4(n-1)}{n-2}\Delta \psi +R\psi=&-\frac{n-1}{n}\tau^{2a}t^{2N}t_0^2\psi^{N-1}\\
	&+\bigg[|\sigma+LW|^2+\bigg(2\max\big\{\|\varphi_0\|_{L^\infty},\,2\big\}-\|\varphi\|_{L^\infty}\bigg)_+\bigg(\|\sigma+LW_0\|^2_{L^\infty}+k_0^2\bigg)\bigg]\psi^{-N-1},\\
	-\frac{1}{2}L^*LW=&\frac{n-1}{n}t^Nt_0\psi^Nd\tau^{a}.
	\endaligned
	\end{equation}}
It follows from the definition of $A(t)$ that
\begin{equation}\label{2. a/t}
\|\psi\|_{L^\infty}\leq A\bigg(t^Nt_0\bigg).
\end{equation}
We will show that $F_{\kappa,2}(t,\varphi)\ne0$. In fact, if $F_{\kappa,2}(t,\varphi)=0$, i.e., $\|\varphi\|_{L^\infty}=b_\kappa>1$, we get by \eqref{2. varphi=tT 2} that
$$
t^{N}t_0=\bigg(\frac{\|\varphi\|_{L^\infty}}{\|\psi\|_{L^\infty}}\bigg)^Nt_0=\bigg(\frac{\max\big\{\|\varphi\|_{L^\infty},\,1\big\}}{\|\psi\|_{L^\infty}}\bigg)^Nt_0.
$$
Combined with $F_{\kappa,1}(t,\varphi)\le 0$, we obtain
$$
t^Nt_0\geq \kappa^{-N} t_0.
$$
However, by \eqref{2. a/t} and the definition of $b_k$, this leads to a contradiction that
$$
b_\kappa=\|\varphi\|_{L^\infty}\leq \|\psi\|_{L^\infty}\leq A\big(t^Nt_0\big)<b_\kappa.
$$
Therefore, $F_{\kappa,2}\ne 0$ as claimed and hence by \eqref{2. F=0 2} we have  
\begin{equation}\label{2. t=kappa}
F_{\kappa,1}(t,\varphi)=\frac{\|\psi\|_{L^\infty}}{\max\big\{\|\varphi\|_{L^\infty},\,1\big\}}-\kappa=0.
\end{equation}
Now analysis similar to that in the proof of Theorem \ref{theorem-nonuniqueness 1} shows that  
\begin{equation}\label{2. varphi > 2max}
\|\varphi\|_{L^\infty} > 2\max\big\{\|\varphi_0\|_{L^\infty},\,2\big\}>1.
\end{equation}
It follows by \eqref{2. varphi=tT 2} and \eqref{2. t=kappa} that
\begin{equation}\label{2. 1/t=kappa}
\frac{1}{t}=\frac{\|\psi\|_{L^\infty}}{\|\varphi\|_{L^\infty}}=\kappa.
\end{equation}
Taking \eqref{2. varphi > 2max}-\eqref{2. 1/t=kappa} into \eqref{2. contradiction-eq 2.1}, we obtain
$$
\aligned
\frac{4(n-1)}{n-2}\Delta \psi +R\psi=&-\frac{n-1}{n}\kappa^{-2N}t_0^2\tau^{2 a}\psi^{N-1}+|\sigma+LW|^2\psi^{-N-1},\\
-\frac{1}{2}L^*LW=&\frac{n-1}{n}\kappa^{-N}t_0\psi^Nd\tau^{a}.
\endaligned
$$
Since $\|\psi\|_{L^{\infty}}>\kappa>c_1$ and since $\kappa$ is an arbitrary constant satisfying \eqref{2.kappa bound blow}, the theorem follows.
\end{proof}

\textbf{Acknowledgments} The author wishes to express his gratitude to Romain Gicquaud and David Maxwell for their advice and helpful discussions. This research is funded by Vietnam National Foundation for Science and Technology Development (NAFOSTED) under grant number 101.02-2016.22. 

\end{document}